\documentclass{eajam}
\renewcommand\thisnumber{x}
\renewcommand\thisyear {200x}
\renewcommand\thismonth{xxx}
\renewcommand\thisvolume{xx}

\usepackage{charter}
\usepackage[charter]{mathdesign}

\begin{document}


\markboth{Sun, Nie, Deng}{Reduced basis method for fractional PDEs}
\title{A reduced finite element formulation for space fractional partial differential equation}

\author[AUTHOR1, AUTHOR2 and AUTHOR3]{Jing Sun, Daxin Nie and Weihua Deng\corrauth}
\address{School of Mathematics and Statistics, Gansu Key Laboratory of Applied Mathematics and Complex Systems, Lanzhou University, Lanzhou 730000, P.R. China.}
\email{{\tt Sunj2015@lzu.edu.cn} (Jing Sun), {\tt ndx1993@163.com} (Daxin Nie), {\tt dengwh@lzu.edu.cn} (Weihua Deng)}

%
%
%

\begin{abstract}
Applying proper orthogonal decomposition to a usual finite element (FE) formulation for space fractional partial differential equation, we get a reduced FE model, which greatly reduces the complexity of computation. Then, the stability analysis and error estimate for the reduced model are presented. Finally, we verify the effectiveness of the algorithm by numerical experiments.
\end{abstract}



\maketitle

\section{ Introduction}
\label{sec1} In recent years, fractional partial differential equations (FPDEs) have become a hot research topic with wide applications in many fields, such as, physics \cite{Metaler2000}, chemistry \cite{Yuste2004}, finance \cite{Picozzi2002}, and so on. Comparing with the classical partial differential equations (PDEs), finding the exact solutions of FPDEs is much more challenging or the solutions themselves are very complicated, being expressed by transcendental functions/infinite series. Developing the numerical methods for the FPDEs naturally attracts the interests of scholars \cite{Cheng2015,Deng2008,Deng2013,DengHes2013,Li2010,LiuDu2015}. Because of the non-locality of fractional operators, one of the key issues is on how to alleviate computational loads and resource demands, especially for the space FPDEs. In practical problems, it is not only the accuracy of the model that matters, but the computational efficiency of the model is likewise critical \cite{Hesthaven}. The goal of this paper is to develop the basic formulation for reduced basis method for space FPDEs to balance the accuracy and efficiency.

The central idea of the reduced basis approach is the identification of a suitable problem dependent basis from the snapshots to effectively represent the solutions to FPDEs, i.e., searching the most representative snapshots and determining when the basis is sufficiently rich. One of the sampling strategies is to make the singular value decomposition of a large number of snapshots, namely, the so-called proper orthogonal decomposition (POD),
which is put forward in the context of turbulence by Lumley \cite{Lumley1967}. Afterwards, POD has been successfully applied in various fields including pattern recognition \cite{Fukunaga}, coherent structures \cite{Sirovich1,Sirovich2,Sirovich3}, control theory \cite{Atwell,Kunisch1999}, and model reduction for PDEs. Some numerical methods combined with the POD have been developed; among them, combining the POD with Galerkin to solve the parabolic equation and fluid dynamics equation is discussed in \cite{Kunisch2003Galerkin,Kunisch2001Galerkin}; Ref. \cite{Luo2009Finite,luo2013a,Luo2011A,Luo2012A} incorporates POD with finite difference, finite element, finite volume to solve classical parabolic problems, Navier-Stokes equations, solute transport problem and so on; Ref. \cite{Liu2016} applies the POD to the finite element format to solve the time FPDEs; all of them could reduce the computation and memory loads after using the POD. To the best of our knowledge, there is no research works on combining the POD and finite element method for space FPDEs.

In this paper, we get a reduced model based on POD and finite element methods for the following problem: Find $u=u(x,y,t)$ satisfying
\begin{equation}\label{equation2D}
 \left\{
 \begin{aligned}
 &\frac{\partial u(x,y,t)}{\partial t}-\frac{\partial ^{\alpha} u(x,y,t)}{\partial |x|^{\alpha}}-\frac{\partial ^{\beta} u(x,y,t)}{\partial |y|^{\beta}}=f(x,y,t)~~~~~~~~~~~~~(x,y,t)\in \Omega\times(0,T),\\
 &u(x,y,0)=g(x,y)~~~~~~~~~~~~~~~~~~~~~~~~~~~~~~~~~~~~~~~~~~~~~~~~~~~~~~~~~~~~~~~~~~~~~~~~~~~~~~~~~~(x, y)\in \Omega,\\
 &u(x,y,t)=0~~~~~~~~~~~~~~~~~~~~~~~~~~~~~~~~~~~~~~~~~~~~~~~~~~~~~~~~~~~~~~~~~~~(x, y)\in \mathbb{R}^2\backslash \Omega, ~~t\in (0,T),
 \end{aligned}
 \right.
 \end{equation}
with $\Omega=(0,1)\times(0,1)$, $1< \alpha, \beta< 2$; and $\frac{\partial ^{\alpha} u}{\partial |x|^{\alpha}}$ denotes the Riesz fractional derivative, being defined by
\begin{equation}
\frac{\partial ^{\alpha} u}{\partial |x|^{\alpha}}=-\frac{1}{2\cos(\frac{\alpha\pi}{2})}(~_{-\infty}D_{x}^{\alpha}u+~_{x}D_{\infty}^{\alpha}u),
\end{equation}
where $~_{-\infty}D_{x}^{\alpha}u$ and $_{x}D_{\infty}^{\alpha}u$ are the left- and right-sided Riemann-Liouville derivatives, respcetively. Since ${\rm supp}(u)\subset\Omega$, we have $_{-\infty}D_x^{\alpha}u(x,y,t)$ =$~_0D_x^{\alpha}u(x,y,t)$, $_{-\infty}D_y^{\alpha}u(x,y,t)$ =$~_0D_y^{\alpha}u(x,y,t)$, $(x,y)\in \Omega$.

For Eq. (\ref{equation2D}), Ref. \cite{Bu2014} provides a finite element method to solve it numerically. Its FE scheme shows that: the stiffness matrix is not sparse since the non-locality of the operator; a finer subdivision is needed to guarantee the accuracy of numerical scheme, leading to the great increase of the memory requirement and the time cost.
To overcome this problem, we combine the POD with the finite element method to solve Eq. (\ref{equation2D}), namely, we reconstruct the POD basis in the least square sense by snapshots which are taken at uniform intervals from the solutions in the general FE, and only $d$ POD basis functions are needed when resolve Eq. (\ref{equation2D}), where $d$ is the number of the first few maximal eigenvalues of matrix $G$ ($G$ is also called the correlation matrix). So the degrees of freedom are reduced and the computing time is also greatly saved. The present method can be considered as an improvement of the classical finite element method.

This paper is organized as follows. In Sec. 2, some preliminaries needed in the paper are presented. Sec. 3 briefly recalls the classical FE method for Eq. (\ref{equation2D}). In Sec. 4, we choose the FE solutions as the snapshots to construct the POD basis in a certain least squares optimal sense and establish the reduced FE scheme based on POD. In Sec. 5, we give the stability analysis and error estimate for the reduced FE scheme. In the last Section, we demonstrate the effectiveness of the model by numerical experiments.
\section{Preliminaries}
We provide the preliminary knowledge in this section.

\begin{definition}[\cite{Podlubny1998Fractional}]\label{def1}
 The left- and right-sided Riemann-Liouville fractional integrals of order $\mu\,(\mu>0)$ are defined by
\begin{equation*}
_{-\infty}I^{\mu}_xu(x)=\frac{1}{\Gamma(\mu)}\int^x_{-\infty}(x-\xi)^{\mu-1}u(\xi)d\xi
\end{equation*}
and
\begin{equation*}
_xI^{\mu}_{\infty}u(x)=\frac{1}{\Gamma(\mu)}\int^{\infty}_x(\xi-x)^{\mu-1}u(\xi)d\xi.
\end{equation*}
\end{definition}
\begin{definition}[\cite{Podlubny1998Fractional}]\label{def2}
The left- and right-sided Riemann-Liouville fractional derivatives of order $\mu\,(\mu>0)$ are described as
\begin{equation*}
_{-\infty}D^{\mu}_xu(x)=\frac{1}{\Gamma(n-\mu)}\frac{d^n}{dx^n}\int^x_{-\infty}(x-\xi)^{n-\mu-1}u(\xi)d\xi
\end{equation*}
and
\begin{equation*}
_xD^{\mu}_{\infty}u(x)=\frac{(-1)^n}{\Gamma(n-\mu)}\frac{d^n}{dx^n}\int^{\infty}_x(\xi-x)^{n-\mu-1}u(\xi)d\xi,
\end{equation*}
where $n-1<\mu<n$.
\end{definition}
\begin{lemma}[\cite{Zhang2010}]
If $0<\mu<1$, $\mu\neq \frac{1}{2}$, $a, b\in\mathbb{R}$, $u, v\in H_0^\mu(a,b)$, then there exist
\begin{equation*}
  (~_aD^{2\mu}_xu,v)=(u,~_xD^{2\mu}_bv)=(~_aD^{\mu}_xu,~_xD^{\mu}_bv)
\end{equation*}
and
\begin{equation*}
  (~_aD^{2\mu}_xu,u)=(u,~_xD^{2\mu}_bu)=(~_aD^{\mu}_xu,~_xD^{\mu}_bu)=\cos(\mu\pi)\|~_aD^{\mu}_xu\|^2.
\end{equation*}
\end{lemma}

\begin{definition}[\cite{Ervin2006}]
For $0\leq\mu<\infty$, we define the space
\begin{equation*}
H^{\mu}(\mathbb{R}):=\{u\mid u \in L^{2}(\mathbb{R}),(1+|\omega|^{2})^{\frac{\mu}{2}}\hat{u}(\omega)\in L^{2}(\mathbb{R})\}
\end{equation*}
with the norm
\begin{equation*}
\|u\|_{H^{\mu}(\mathbb{R})}:=(\|u\|^{2}+|u|_{\mu,\mathbb{R}})^{\frac{1}{2}} ~~~\forall u\in H^{\mu}(\mathbb{R}),
\end{equation*}
where $|u|_{\mu,\mathbb{R}}:=\||\omega|^{\mu}\hat u\|$, and $\hat{u}$ denotes Fourier transform of $u$.
\end{definition}
\begin{definition}[\cite{Ervin2006}]
For $-\infty \leq a < b \leq \infty$, we define
\begin{equation*}
H^{\mu}(a,b):=\{v\mid_{(a,b)}~|v\in H^{\mu}(\mathbb{R})\}
\end{equation*}
with the norm
\begin{equation*}
\|v\|_{H^{\mu}(a,b)}:=\inf_{\substack{
                                \tilde{v}\in H^{\mu}(\mathbb{R})\\
                                \tilde{v}\mid_{(a,b)}=v}} \|\tilde{v} \|_{H^\mu(\mathbb{R})}~~~~\forall v\in~H^\mu(a,b).
\end{equation*}
\end{definition}
Furthermore, let $\mathcal{D}(a,b)$ be the set of $C^{\infty}$ functions with compact support in $(a,b)$; and $H_{0}^{\mu}(a,b)$ is the closure of $\mathcal{D}(a,b)$ with respect to $\|\cdot\|_{H^{\mu}(a,b)}$.
\begin{lemma}[\cite{Ervin2006}]
If $u \in H^\mu_0(\mathbb{R})$, $\mu\neq n+1/2$, $n \in \mathbb{N}$, we have\\
\begin{equation*}
\parallel u\parallel\leq C\mid u\mid_\mu,
\end{equation*}
where $C>0$.
\end{lemma}
\begin{lemma}[\cite{luo2013a} Discrete Gr\"onwall Lemma]
If $\{a_n\}$, $\{b_n\}$, $\{c_n\}$ are three positive sequences, and $\{c_n\}$ is monotone, that satisfy
\begin{equation}
a_n+b_n\leq c_n+\lambda \sum\limits_{i=0}^{n-1}a_i, ~~~~~~~~\lambda>0,~a_0+b_0\leq c_0,
\end{equation}
then $a_n+b_n\leq c_n \exp(n\lambda)$,~~~$n>0$.
\end{lemma}

\section{Recall of classical FE formulation}
\label{sec2}

To solve Eq. $(\ref{equation2D})$ numerically, we use FE method to discretize the spatial variable and the backward Euler to the time variable. Firstly, we take the FE space as
\begin{align*}
X_{h}=\{v_{h} \in   H_0^{\frac{\alpha}{2}}(\Omega)\bigcap H_0^{\frac{\beta}{2}}(\Omega)\bigcap C^{0}(\Omega); v_{h} \in P_{m}(K)\,\forall K \in \Im_{h}\},
\end{align*}
where $m\geq 1$, $P_{m}(K)$ is taken as the piecewise polynomials of degree $\leq m$ of $K$, and  $\{\Im_{h}\}$ stands for a uniformly regular family of triangulation of $\Omega$. Then the semi-discrete FE formulation can be written as:
For every $t\in (0,T)$, find $u_h\in X_{h}$ such that
\begin{align}\label{eq:finiteform2D}
\begin{split}
	\left(\frac{\partial u_{h}}{\partial t},v_h\right)-\left(\frac{\partial ^{\alpha} u_h}{\partial |x|^{\alpha}},v_h\right)-\left(\frac{\partial ^{\beta} u_h}{\partial |y|^{\beta}},v_h\right)= (f,v_h) \quad \forall v_h\in X_{h},
\end{split}
\end{align}
namely,
\begin{align*}
\left(\frac{\partial u_{h}}{\partial t},v_h\right)&+\frac{1}{2\cos(\frac{\alpha\pi}{2})}(~_{0}D_{x}^{\alpha}u_h,v_h)+\frac{1}{2\cos(\frac{\alpha\pi}{2})}(~_{x}D_{1}^{\alpha}u_h,v_h)\\
&+\frac{1}{2\cos(\frac{\beta\pi}{2})}(~_{0}D_{y}^{\beta}u_h,v_h)+\frac{1}{2\cos(\frac{\beta\pi}{2})}(~_{y}D_{1}^{\beta}u_h,v_h)= (f,v_h) \quad \forall v_h\in X_{h}.
\end{align*}
Next, let $N$ be an integer, $\tau=\frac{T}{N}$ be the time step size, and $t_{n}=n\tau\,(0\leq n\leq N)$. Then the fully discrete FE formulation is written as
\begin{equation}\label{eq:2Ddiscretescheme}
\begin{aligned}
(u^{n}_h,v_h)&+\tau a(u^n_h,v_h) =\tau(f,v_h)+(u_h^{n-1},v_h),
\end{aligned}
\end{equation}
where $C_\alpha =\frac{1}{2\cos(\alpha\pi/2)}$, $C_\beta =\frac{1}{2\cos(\beta\pi/2)}$, and
\begin{equation}\label{aaaa}
\begin{aligned}
a(u_h,v_h)=&C_{\alpha}(~_{0}D_{x}^{\frac{\alpha}{2}}u_h,~_{x}D_{1}^{\frac{\alpha}{2}}v_h)+C_{\alpha}(~_{x}D_{1}^{\frac{\alpha}{2}}u_h,~_{0}D_{x}^{\frac{\alpha}{2}}v_h)\\
&+ C_{\beta}(~_{0}D_{y}^{\frac{\beta}{2}}u_h,~_{y}D_{1}^{\frac{\beta}{2}}v_h)+ C_{\beta}(~_{y}D_{1}^{\frac{\beta}{2}}u_h,~_{0}D_{y}^{\frac{\beta}{2}}v_h).
\end{aligned}
\end{equation}
\begin{remark}
The existence and uniqueness of the solutions to Eq. (\ref{eq:2Ddiscretescheme}) can be obtained by Lax-Milgram theorem\cite{Bu2014}.
\end{remark}

When the source term $f$, the triangulation parameter $h$, the time step increment $\tau$, and the FE space $X_h$ are given, we can get an ensemble of solutions $\{{u_{h}^{n}}\}_{n=1}^{N}$ for Eq. (\ref{equation2D}). Then we choose $L\,(L\ll N,~{\rm usually}~L=20)$ instantaneous solutions $u_{h}^{n_i} (1\leq n_1\leq n_2\leq\cdots\leq n_L \leq N)$ at an uniform interval from $N$ solutions $\{{u_{h}^{n}}\}_{n=1}^{N}$ for Eq. (\ref{equation2D}), being referred to as snapshots of the POD.

\begin{remark}
For the practical problems, one can get the snapshots by drawing samples from experiments or from the past data information. After obtaining the ensemble of snapshots from previous prediction, one can construct the POD basis and the finite element space $X_h$ is substituted with the subspace generated by the POD basis so as to get the reduced formulation.

\end{remark}
\section{Construction of the reduced FE formulation by POD}
By POD, we build the reduced FE formulation for Eq. (\ref{equation2D}).
\subsection{Generation of POD bases }
 Suppose that $U_{i}(x,y)=u^{n_i}_{h}(x,y)\,(1\leq i\leq L)$ and at least one of which is assumed to be nonzero. Let
\begin{equation}
  V={\rm span}\{U_1,\cdots,U_L\},
\end{equation}
and $\{\psi_{j}\}_{j=1}^l$ stands for an orthonormal basis of $V$ with $l=\dim V\leq L$. Since  $V\subseteq H_{0}^{\frac{\alpha}{2}}(\Omega)\bigcap H_0^{\frac{\beta}{2}}(\Omega)$ and $\{\psi_{j}\}_{j=1}^l$ is an orthonormal basis, we have
\begin{equation}\label{(2D1.1)}
U_i=\sum\limits_{j=1}^l\left(U_i,\psi_j\right)_w\psi_j,~~~~i=1,2,\cdots, L,
\end{equation}
where
\begin{align*}
\left(U_i,\psi_j\right)_w:=&C_\alpha\left(~_{0}D_{x}^{\frac{\alpha}{2}}U_i,~_{x}D_{1}^{\frac{\alpha}{2}}\psi_j\right)+C_\alpha\left(~_{x}D_{1}^{\frac{\alpha}{2}}U_i,~_{0}D_{x}^{\frac{\alpha}{2}}\psi_j\right)\\
                &+C_\beta\left(~_{0}D_{y}^{\frac{\beta}{2}}U_i,~_{y}D_{1}^{\frac{\beta}{2}}\psi_j\right)+C_\beta\left(~_{y}D_{1}^{\frac{\beta}{2}}U_i,~_{0}D_{y}^{\frac{\beta}{2}}\psi_j\right).
\end{align*}

\begin{definition}[\cite{Luo2012A}]
 The method of POD consists in finding the orthonormal basis $\psi_{j}$ $(j=1,\cdots,L)$, such that for every $d\,(1\leq d\leq l)$, the mean  square error between the elements $U_{i}$  and  the corresponding $d$-th partial sum of $(\ref{(2D1.1)})$ is minimized on average, i.e.,
\begin{equation}\label{eq:2Doptimization}
\min\limits_{\{{\psi_j}\}_{j=1}^{d}}\frac{1}{L}\sum\limits_{i=1}^{L}\left\|U_{i}-\sum_{j=1}^{d}(U_i,\psi_j)_w\psi_j\right\|_{w}^{2},
\end{equation}
subject to
\begin{equation}
\left(\psi_i, \psi_j\right)_w=\delta_{ij},~~~~1\leq i\leq d ,~~1\leq j\leq i,
\end{equation}
where
\begin{equation}
\begin{aligned}
 \|U_{i}\|_w^{2}                   :=&C_\alpha\left(~_{0}D_{x}^{\frac{\alpha}{2}}U_i,~_{x}D_{1}^{\frac{\alpha}{2}}U_i\right)+C_\alpha\left(~_{x}D_{1}^{\frac{\alpha}{2}}U_i,~_{0}D_{x}^{\frac{\alpha}{2}}U_i\right) \\
 &+C_\beta\left(~_{0}D_{y}^{\frac{\beta}{2}}U_i,~_{y}D_{1}^{\frac{\beta}{2}}U_i\right)+C_\beta\left(~_{y}D_{1}^{\frac{\beta}{2}}U_i,~_{0}D_{y}^{\frac{\beta}{2}}U_i\right).
 \end{aligned}
 \end{equation}
\end{definition}

\begin{remark}
It is easy to verify that the function space defined using $\|\cdot\|_w$ is equivalent to the space $H_{0}^{\frac{\alpha}{2}}(\Omega)\bigcap H_0^{\frac{\beta}{2}}(\Omega)$.
\end{remark}

By the definition (\ref{(2D1.1)}) and the orthonormality of $\psi_i$, we can rewrite (\ref{eq:2Doptimization}) as
\begin{equation}\label{www1}
\begin{aligned}
\frac{1}{L}\sum\limits_{i=1}^{L}\left\|U_{i}-\sum_{j=1}^{d}\left(U_i,\psi_j\right)_w\psi_j\right\|_{w}^{2}&=\frac{1}{L}\sum_{i=1}^{L}\left\|\sum_{j=d+1}^{l}\left(U_i,\psi_j\right)_w\psi_j\right\|_{w}^{2}\\
&=\sum_{j=d+1}^{l}\left[\frac{1}{L}\sum_{i=1}^{L}\left|\left(U_i,\psi_j\right)^{2}_w\right|\right].
\end{aligned}
\end{equation}
Moreover,
\begin{align*}
\sum_{j=1}^{l}\left[\frac{1}{L}\sum_{i=1}^{L}\left|\left(U_i,\psi_j\right)_w^{2}\right|\right]=\frac{1}{L}\sum_{i=1}^{L}\left\|\sum_{j=1}^{l}\left(U_i,\psi_j\right)_w\psi_j\right\|_{w}^{2}=\frac{1}{L}\sum_{i=1}^{L}\left\|U_{i}\right\|_{w}^{2}
\end{align*}
is with a fixed value.
Thus, in order to make (\ref{www1}) minimum, one only needs to find the orthonormal basis $\psi_j\,(j=1,2\ldots,l)$ such that
\begin{equation}\label{maxx}
\max\limits_{\{{\psi_j}\}_{j=1}^{d}}\sum\limits_{j=1}^{d}\left[\frac{1}{L}\sum_{i=1}^{L}\left|(U_i,\psi_j)_{w}^{2}\right|\right],
\end{equation}
subject to
\begin{equation}\label{condition}
(\psi_i, \psi_j)_w=\delta_{ij},~~~~1\leq i\leq d ,~~1\leq j\leq i.
\end{equation}
Following \cite{Luo2012A,Luo2011A}, to solve (\ref{maxx})-(\ref{condition}), one can start from  finding a function (or the so-called POD basis element) $\psi$ such that it maximizes \begin{equation}\label{maxxxx}
 \frac{1}{L}\sum\limits_{i=1}^{L}\left|(U_i,\psi)_{w}^{2}\right|
\end{equation}
satisfying $(\psi,\psi)_{w}=1$. Here, we choose $\psi$ having the form: $\psi=\sum\limits_{i=1}^{L}a_{i}U_{i}$, where $a_i$ is determined to make (\ref{maxxxx}) maximum. Then, define the operators
\begin{equation}
K\left((x,y),(x',y')\right)=\frac{1}{L}\sum_{i=1}^{L}U_{i}(x,y)U_{i}(x',y')
\end{equation}
and
\begin{equation}\label{eq:2DRproj}
\begin{aligned}
R\psi= &C_\alpha\int\int_{\Omega}~_{0}D_{x'}^{\frac{\alpha}{2}}K\left((x,y),(x',y')\right)~_{x'}D_{1}^{\frac{\alpha}{2}}\psi(x',y')dx'dy'\\
&+C_\alpha\int\int_{\Omega}~_{0}D_{x'}^{\frac{\alpha}{2}}\psi(x',y')~_{x'}D_{1}^{\frac{\alpha}{2}}K\left((x,y),(x',y')\right)dx'dy'\\
&+C_\beta\int\int_{\Omega}~_{0}D_{y'}^{\frac{\beta}{2}}K\left((x,y),(x',y')\right)~_{y'}D_{1}^{\frac{\beta}{2}}\psi(x',y')dx'dy'\\
&+C_\beta\int\int_{\Omega}~_{0}D_{y'}^{\frac{\beta}{2}}\psi(x',y')~_{y'}D_{1}^{\frac{\beta}{2}}K\left((x,y),(x',y')\right)dx'dy',
\end{aligned}
\end{equation}
where $R: H_0^{\frac{\alpha}{2}}(\Omega)\bigcap H_0^{\frac{\beta}{2}}(\Omega) \longrightarrow H_0^{\frac{\alpha}{2}}(\Omega) \bigcap H_0^{\frac{\beta}{2}}(\Omega)$.
Direct calculation leads to
\begin{align*}
(R\psi,\psi)_{w}=&C_\alpha\int\int_{\Omega}~_{0}D_{x}^{\frac{\alpha}{2}}R\psi~_{x}D_{1}^{\frac{\alpha}{2}}\psi dxdy+C_\alpha\int\int_{\Omega}~_{0}D_{x}^{\frac{\alpha}{2}}\psi ~_{x}D_{1}^{\frac{\alpha}{2}}R\psi dxdy\\
&+C_\beta\int\int_{\Omega}~_{0}D_{y}^{\frac{\beta}{2}}R\psi~_{y}D_{1}^{\frac{\beta}{2}}\psi dxdy+C_\beta\int\int_{\Omega}~_{0}D_{y}^{\frac{\beta}{2}}\psi~ _{y}D_{1}^{\frac{\beta}{2}}R\psi dxdy\\
=&\frac{1}{L}\sum_{i=1}^{L}\left|(U_i,\psi)^{2}_{w}\right|.
\end{align*}
Furthermore, according to
\begin{equation}
(R\phi,\psi)_{w}=(R\psi,\phi)_{w},
\end{equation}
it can be got that $R$ is a nonnegative symmetric operator on $H_{0}^{\frac{\alpha}{2}}(\Omega)\bigcap H_{0}^{\frac{\beta}{2}}(\Omega)$.
So we transform the problem (\ref{maxx})-(\ref{condition}) to find the largest eigenvalue for the problem
\begin{equation}\label{eigenvalue}
R\psi=\lambda\psi ~~~~~~~~~~~\text{subject~to}~(\psi,\psi)_{w}=1.
\end{equation}
According to the definition of $R$, $K$ and $\psi$, (\ref{eigenvalue}) becomes
\begin{equation}\label{leftside}
\begin{aligned}
&\sum_{i=1}^{L}U_i(x,y)\sum_{j=1}^{L}a_j\left[\frac{ C_\alpha}{ L}\int\int_{\Omega}~_{0}D_{x'}^\frac{\alpha}{2}U_i(x',y')~_{x'}D_{1}^\frac{\alpha}{2}\left(\sum_{j=1}^{L}U_j(x',y')\right)dx'dy'\right.\\
&+\frac{ C_\alpha}{ L}\int\int_{\Omega}~_{x'}D_{1}^\frac{\alpha}{2}U_i(x',y')~_{0}D_{x'}^\frac{\alpha}{2}\left(\sum_{j=1}^{L}U_j(x',y')\right)dx'dy'\\
&+\frac{ C_\beta}{ L}\int\int_{\Omega}~_{0}D_{y'}^\frac{\beta}{2}U_i(x',y')~_{y'}D_{1}^\frac{\beta}{2}\left(\sum_{j=1}^{L}U_j(x',y')\right)dx'dy'\\
&\left.+\frac{ C_\beta}{ L}\int\int_{\Omega}~_{y'}D_{1}^\frac{\beta}{2}U_i(x',y')~_{0}D_{y'}^\frac{\beta}{2}\left(\sum_{j=1}^{L}U_j(x',y')\right)dx'dy'\right]\\
=&\lambda\sum_{i=1}^{L}a_iU_i(x,y),
\end{aligned}
\end{equation}
Simplifying (\ref{leftside}) further, one can get
\begin{align*}
&\sum\limits_{j=1}^{L}a_j\left[\frac{C_\alpha }{ L}\int\int_{\Omega}~_{0}D_{x'}^\frac{\alpha}{2}U_i(x',y')~_{x'}D_{1}^\frac{\alpha}{2}U_j(x',y')dx'dy'\right.\\
&+\frac{ C_\alpha}{ L}\int\int_{\Omega}~_{x'}D_{1}^\frac{\alpha}{2}U_i(x',y')~_{0}D_{x'}^\frac{\alpha}{2}U_j(x',y')dx'dy'\\
&+\frac{C_\beta }{ L}\int\int_{\Omega}~_{0}D_{y'}^\frac{\beta}{2}U_i(x',y')~_{y'}D_{1}^\frac{\beta}{2}U_j(x',y')dx'dy'\\
&\left.+\frac{C_\beta }{ L}\int\int_{\Omega}~_{y'}D_{1}^\frac{\beta}{2}U_i(x',y')~_{0}D_{y'}^\frac{\beta}{2}U_j(x',y')dx'dy'\right]=\lambda a_i.
\end{align*}

Denote
\begin{align*}
G_{ij}=&\frac{C_\alpha }{L}\int\int_{\Omega}~_{0}D_{x'}^\frac{\alpha}{2}U_i(x',y')~_{x'}D_{1}^\frac{\alpha}{2}U_j(x',y')dx'dy'\\
&+\frac{C_\alpha }{ L}\int\int_{\Omega}~_{x'}D_{1}^\frac{\alpha}{2}U_i(x',y')~_{0}D_{x'}^\frac{\alpha}{2}U_j(x',y')dx'dy'\\
&+\frac{C_\beta }{ L}\int\int_{\Omega}~_{0}D_{y'}^\frac{\beta}{2}U_i(x',y')~_{y'}D_{1}^\frac{\beta}{2}U_j(x',y')dx'dy'\\
&+\frac{C_\beta }{ L}\int\int_{\Omega}~_{y'}D_{1}^\frac{\beta}{2}U_i(x',y')~_{0}D_{y'}^\frac{\beta}{2}U_j(x',y')dx'dy'.
\end{align*}
Then, the eigenvalue problem can be transformed to
 \begin{equation}
 G\mathbf{v}=\lambda \mathbf{v},~~~~~~~~~\mathbf{v}=[a_1,\cdots,a_L]^T.
 \end{equation}
Since the matrix $G$ is a nonnegative Hermitian matrix with the rank $l$,  it has a complete set of orthonormal eigenvectors $\mathbf{v}^i=[a_1^i,a_2^i,\cdots,a_L^i],~i=1,2,\cdots, l$, with the corresponding eigenvalues $\lambda_1\geq \lambda_2\geq\cdots\geq \lambda_l>0$.
Thus, the solution to the optimization for (\ref{eq:2Doptimization}) is given by
\begin{equation}
\psi_1=\frac{1}{\sqrt{L\lambda_1}}\sum_{i=1}^{L}a_i^1U_i,
\end{equation}
where $a_i^1~(i=1,2,\cdots,L)$ are the elements of the eigenvector $\mathbf{v}^1$  corresponding to the largest eigenvalue $\lambda_1$.

Similarly, the basis of POD $\psi_k \,(k=2,3,\cdots,l)$ are obtained by using other eigenvectors $\mathbf{v}^k\,(k=2,\cdots,l)$,
\begin{equation}
\psi_k=\frac{1}{\sqrt{L\lambda_k}}\sum_{i=1}^{L}a_i^kU_i,~~~~~k=2,3,\cdots,l.
\end{equation}
By the orthonormality of $\{\mathbf{v}^k: 1\leq k\leq l\}$, there exists
\begin{equation}
\mathbf{v}^k\cdot\mathbf{ v}^{k'}=\sum_{i=1}^{L}a_{i}^{k}a_{i}^{k'}=\left\{
\begin{aligned}
&1,~k=k',\\
&0,~k\neq k',
\end{aligned}
\right.
\end{equation}
Furthermore, one can obtain that
\begin{equation}
\begin{aligned}
(\psi_k,\psi_{k'})_{w}=&C_\alpha\int\int_{\Omega} \left(~_{0}D_{x}^{\frac{\alpha}{2}}\psi_k~_xD_1^{\frac{\alpha}{2}}\psi_{k'}+~_{0}D_{x}^{\frac{\alpha}{2}}\psi_{k'}~_xD_1^{\frac{\alpha}{2}}\psi_{k}\right)dxdy\\
&+C_\beta\int\int_{\Omega} \left(~_{0}D_{y}^{\frac{\beta}{2}}\psi_k~_yD_1^{\frac{\beta}{2}}\psi_{k'}+~_{0}D_{y}^{\frac{\beta}{2}}\psi_{k'}~_yD_1^{\frac{\alpha}{2}}\psi_{k}\right)dxdy\\
=&C_\alpha\int\int_{\Omega}\left(~_{0}D_{x}^{\frac{\alpha}{2}}\left(\frac{1}{\sqrt{L\lambda_k}}\sum_{i=1}^{L}a_i^{k}U_i\right)~_{x}D_{1}^{\frac{\alpha}{2}}\left(\frac{1}{\sqrt{L\lambda_{k'}}}\sum_{i=1}^{L}a_i^{k'}U_i\right)\right.\\
&\left.+~_{0}D_{x}^{\frac{\alpha}{2}}\left(\frac{1}{\sqrt{L\lambda_{k'}}}\sum_{i=1}^{L}a_i^{k'}U_i\right)~_{x}D_{1}^{\frac{\alpha}{2}}\left(\frac{1}{\sqrt{L\lambda_k}}\sum_{i=1}^{L}a_i^{k}U_i\right)\right)dxdy\\
&+C_\beta\int\int_{\Omega}\left(~_{0}D_{y}^{\frac{\beta}{2}}\left(\frac{1}{\sqrt{L\lambda_k}}\sum_{i=1}^{L}a_i^{k}U_i\right)~_{y}D_{1}^{\frac{\beta}{2}}\left(\frac{1}{\sqrt{L\lambda_{k'}}}\sum_{i=1}^{L}a_i^{k'}U_i\right)\right.\\
&\left.+~_{0}D_{y}^{\frac{\beta}{2}}\left(\frac{1}{\sqrt{L\lambda_{k'}}}\sum_{i=1}^{L}a_i^{k'}U_i\right)~_{y}D_{1}^{\frac{\beta}{2}}\left(\frac{1}{\sqrt{L\lambda_k}}\sum_{i=1}^{L}a_i^{k}U_i\right)\right)dxdy\\
=&\frac{1}{\sqrt{\lambda_{k}\lambda_{k'}}}\sum_{i=1}^{L}a_i^k\sum_{j=1}^{L}G_{ij}a_j^{k'}\\
=&\frac{1}{\sqrt{\lambda_{k}\lambda_{k'}}}\mathbf{v^{k}} G\mathbf{ v^{k'}}\\
=&\frac{1}{\sqrt{\lambda_{k}\lambda_{k'}}}\mathbf{v^{k}} \lambda_{k'} \mathbf{v^{k'}}\\
=&\left\{
\begin{aligned}
1~~k=k'\\
0~~k\neq k'.
\end{aligned}
\right.
\end{aligned}
\end{equation}
So the POD basis $\{\psi_1,\psi_2\ldots,\psi_l\}$ forms an orthonormal set. Next, we give a theorem for proving that the basis obtained by the above POD method is optimal.
\begin{theorem}
The POD basis $\{\psi_i\}_{i=1}^{l}$ is an optimal one.
\end{theorem}
\begin{proof}
 Assume that $\{\psi_i\}_{i=1}^{l}$ isn't the optimal orthonormal basis. Then we denote the optimal one as $\{\phi_i\}_{i=1}^{l}$. Namely,
\begin{equation}
\begin{aligned}
\psi=(\psi_1,\psi_2,\ldots,\psi_l)^T,\\
\phi=(\phi_1,\phi_2,\ldots,\phi_l)^T.\\
\end{aligned}
\end{equation}
Since there is an unitary matrix between two different orthonormal bases, then
\begin{equation}
  \phi=A\psi,
\end{equation}
where $A$ is an unitary matrix; and
\begin{equation}
  \begin{pmatrix}
    \phi_{k_1}\\
    \vdots\\
    \phi_{k_d}
  \end{pmatrix}
  =\begin{pmatrix}
    A_{k_1}\\
    \vdots\\
    A_{k_d}
  \end{pmatrix}
  \psi,
\end{equation}
where $A_i$ stands for the $i$-th line of matrix $A$ and $d\leq l$.
$R$ is defined by $(\ref{eq:2DRproj})$, and denote the eigenvalues $\lambda_1>\lambda_2>\ldots >\lambda_l$ for $\psi$. Then
\begin{equation}
  \begin{aligned}
\sum_{i=1}^{d}(R\psi_{i},\psi_{i})=&\sum_{i=1}^{d}(\lambda_i\psi_i,\psi_i)\\
    =&\sum_{i=1}^{d}\lambda_i
  \end{aligned}
\end{equation}
and
\begin{equation}
  \begin{aligned}
    \sum_{i=1}^{d}(R\phi_{k_i},\phi_{k_i})=&\sum_{i=1}^{d}(RA_{k_i}\psi,A_{k_i}\psi)\\
    =&\sum_{i=1}^{d}\sum_{j=1}^l a_{k_ij}^2\lambda_j.
  \end{aligned}
\end{equation}
According to the property of unitary matrix, there is $\sum_{i=1}^{d}\sum_{j=1}^l a_{k_ij}^2\lambda_j\leq\sum_{i=1}^{d}\lambda_i$, being contradicted with the optimality of $\{\phi_i\}_{i=1}^{l}$.
Therefore, the POD basis $\{\psi_i\}_{i=1}^{l}$ obtained by the above POD method is optimal one.
\end{proof}

\subsection{Reduced FE formulation based on POD}
Let $W^d={\rm span}\{\psi_1, \psi_2,\ldots, \psi_d\}$. Then $W^d \subset X_h$. Define the projection $P^d$: $X_h \rightarrow W^d$ denoted by (see\cite{Rudin,Luo2012A})
\begin{equation}\label{projection2}
a(P^dU,V_d)=a(U,V_d)~~~~~\forall V_d\in W^d,
\end{equation}
where $a(u,v)$ is defined by (\ref{aaaa}).
According to the theory of linear operator, there is an extension $P^h$: $H_{0}^{\frac{\alpha}{2}}\bigcap H_{0}^{\frac{\beta}{2}}\rightarrow X_h$ such that $P^h|_{X_h}=P^d:\, X_h\rightarrow W^d$ satisfying
\begin{equation}
a(P^hU,V_h)=a(U,V_h)~~~~~\forall V_h\in X_h.
\end{equation}

\begin{theorem}\label{prosp}
When $U \in H^\alpha(\Omega)\bigcap H^\beta(\Omega)$, for every $d\,(1\leq d\leq l)$, the projection operator $P^d$ satisfies
\begin{equation}\label{thm:2Dprojection:eq:1}
\frac{1}{L}\sum_{i=1}^L\left\|~_0D_x^{\frac{\alpha}{2}}\left(U_h^{n_i}-P^dU_h^{n_i}\right)\right\|^2+\frac{1}{L}\sum_{i=1}^L\left\|~_0D_y^{\frac{\beta}{2}}\left(U_h^{n_i}-P^dU_h^{n_i}\right)\right\|^2\leq C\sum_{j=d+1}^l\lambda_j,
\end{equation}
where $U_h^{n_i}$ is the solution of FE scheme.
\end{theorem}
\begin{proof}
Since
 \begin{equation}
    a(U,V_h)=a(P^hU,V_h)~~~~~~\forall V_h\in X_h,
 \end{equation}
we have
  \begin{equation}
     a(U-P^hU,V_h)=0~~~~~~~~~~\forall V_h\in X_h.
  \end{equation}
  Moreover, according to
 \begin{equation}
 \left\|~_0D_x^{\frac{\alpha}{2}}(U-P^hU)\right\|^2+\left\|~_0D_y^{\frac{\beta}{2}}\left(U-P^hU\right)\right\|^2=a\left(U-P^hU,U-P^hU\right)
 \end{equation}
and
  \begin{equation}
    \begin{aligned}
      &a\left(U-P^hU,U-P^hU\right)\\
      =&a\left(U-P^hU,U-V_h\right)+a\left(U-P^hU,V_h-P^hU\right)\\
      =&a\left(U-P^hU,U-V_h\right)\\
      =&C_\alpha \left[\left(~_0D_x^{\frac{\alpha}{2}}\left(U-P^hU\right),~_xD_1^{\frac{\alpha}{2}}\left(U-V_h\right)\right)+\left(~_xD_1^{\frac{\alpha}{2}}\left(U-P^hU\right),~_0D_x^{\frac{\alpha}{2}}\left(U-V_h\right)\right)\right]\\
      &+C_\beta \left[\left(~_0D_y^{\frac{\beta}{2}}\left(U-P^hU\right),~_yD_1^{\frac{\beta}{2}}\left(U-V_h\right)\right)+\left(~_yD_1^{\frac{\beta}{2}}\left(U-P^hU\right),~_0D_y^{\frac{\beta}{2}}\left(U-V_h\right)\right)\right]\\
      \leq&C\left[\left\|~_0D_x^{\frac{\alpha}{2}}\left(U-P^hU\right)\right\| \left\|~_xD_1^{\frac{\alpha}{2}}\left(U-V_h\right)\right\|+\left\|~_xD_1^{\frac{\alpha}{2}}\left(U-P^hU\right)\right\|\left\|~_0D_x^{\frac{\alpha}{2}}\left(U-V_h\right)\right\|\right]\\
      &+C\left[\left\|~_0D_y^{\frac{\beta}{2}}\left(U-P^hU\right)\right\|\left\|~_yD_1^{\frac{\beta}{2}}\left(U-V_h\right)\right\|+\left\|~_yD_1^{\frac{\beta}{2}}\left(U-P^hU\right)\right\|\left\|~_0D_y^{\frac{\beta}{2}}\left(U-V_h\right)\right\|\right]\\
      \leq&C\left(\left\|~_0D_x^{\frac{\alpha}{2}}\left(U-P^hU\right)\right\|+\left\|~_0D_y^{\frac{\beta}{2}}\left(U-P^hU\right)\right\|\right)\left(\left\|~_0D_x^{\frac{\alpha}{2}}\left(U-V_h\right)\right\|+\left\|~_0D_y^{\frac{\beta}{2}}\left(U-V_h\right)\right\|\right),
    \end{aligned}
  \end{equation}
  using
  \begin{equation}
  \begin{aligned}
   &\left\|~_0D_x^{\frac{\alpha}{2}}\left(U-P^hU\right)\right\|^2+\left\|~_0D_y^{\frac{\beta}{2}}\left(U-P^hU\right)\right\|^2
   \\
   &\geq\frac{1}{2} \left(\left\|~_0D_x^{\frac{\alpha}{2}}\left(U-P^hU\right)\right\|+\left\|~_0D_y^{\frac{\beta}{2}}\left(U-P^hU\right)\right\|\right)^2,
 \end{aligned}
  \end{equation}
  we have
  \begin{equation}\label{projectionin}
  \left\|~_0D_x^{\frac{\alpha}{2}}\left(U-P^hU\right)\right\|+\left\|~_0D_y^{\frac{\beta}{2}}\left(U-P^hU\right)\right\|\leq C\left(\left\|~_0D_x^{\frac{\alpha}{2}}(U-V_h)\right\|+\left\|~_0D_y^{\frac{\beta}{2}}\left(U-V_h\right)\right\|\right).
  \end{equation}

If we take $U=U^{n_i}_h$, and let $P^h$ be restricted from $X_h$ to $W^d$, i.e., $P^hU=P^dU_h^{n_i}\in W^d$. Let $V_h=\sum_{i=1}^d\left(U_h^{n_i},\psi_j\right)_w\psi_j\in W^d\subset X_h$. Since
  \begin{equation}
  \begin{aligned}
    \frac{1}{L}\sum_{i=1}^{L}\left\|U_i-\sum_{j=1}^{d}(U_i,\psi_j)_w\psi_j\right\|_w^2&=\frac{1}{L}\sum\limits_{j=d+1}^{l}\sum\limits_{i=1}^{L}\left|(U_i,\psi_j)_{w}^{2}\right|\\
    &=\sum_{j=d+1}^l\lambda_j,
  \end{aligned}
  \end{equation}
according to (\ref{projectionin}), we have
  \begin{equation}
    \begin{aligned}
        &\frac{1}{L}\sum_{i=1}^l\left\|~_0D_x^{\frac{\alpha}{2}}\left(U_h^{n_i}-P^dU_h^{n_i}\right)\right\|^2+\frac{1}{L}\sum_{i=1}^l\left\|~_0D_y^{\frac{\beta}{2}}\left(U_h^{n_i}-P^dU_h^{n_i}\right)\right\|^2\\
        \leq&C\frac{1}{L}\sum_{i=1}^l\left\|~_0D_x^{\frac{\alpha}{2}}\left(U_h^{n_i}-\sum_{j=1}^{d}\left(U_h^{n_i},\psi_j\right)_w\psi_j\right)\right\|^2\\
        &+C\frac{1}{L}\sum_{i=1}^l\left\|~_0D_y^{\frac{\beta}{2}}\left(U_h^{n_i}-\sum_{j=1}^{d}\left(U_h^{n_i},\psi_j\right)_w\psi_j\right)\right\|^2\\
        \leq&C\frac{1}{L}\sum_{i=1}^l\left\|\left(U_h^{n_i}-\sum_{j=1}^{d}\left(U_h^{n_i},\psi_j\right)_w\psi_j\right)\right\|_w^2\\
        \leq&C\sum_{j=d+1}^l \lambda_j.
    \end{aligned}
  \end{equation}
 The proof of $(\ref{thm:2Dprojection:eq:1})$ is completed.
\end{proof}

By using $W^d={\rm span}\left\{\psi_1, \psi_2,\cdots, \psi_d\right\}$, based on POD we obtain the reduced FE formulation: Find $u_d^n\in W^d$ such that
\begin{equation}\label{eq:2DPODscheme}
\left\{
\begin{aligned}
  &\left(u^{n}_d,v_d\right)+\tau C_{\alpha}\left(~_{0}D_{x}^{\frac{\alpha}{2}}u^{n}_d,~_{x}D_{1}^{\frac{\alpha}{2}}v_d\right)+\tau C_{\alpha}\left(~_{x}D_{1}^{\frac{\alpha}{2}}u^{n}_d,~_{0}D_{x}^{\frac{\alpha}{2}}v_d\right)\\
  &~~~~+\tau C_{\beta}\left(~_{0}D_{y}^{\frac{\beta}{2}}u^{n}_d,~_{y}D_{1}^{\frac{\beta}{2}}v_d\right)+\tau C_{\beta}\left(~_{y}D_{1}^{\frac{\beta}{2}}u^{n}_d,~_{0}D_{y}^{\frac{\beta}{2}}v_d\right)\\
  &~~~~=\tau\left(f^n,v_d\right)+\left(u_d^{n-1},v_d\right) ~~~~~~~~~~~~~~~~~~~~~~~~~~~~~~~~~~~~~~~~~~~~~~~~~~~~~\forall~v_d \in W^{d},\\
  &u_d^0=P^du_h^0.
  \end{aligned}
  \right.
\end{equation}

\section{Stability analysis and error estimates for the reduced FE formulation}
Now, we perform the numerical stability analysis and provide the error estimates.
\begin{theorem}\label{thm2Derror}
Let $u_{h}^{n}\in X_{h}$ be the finite element solution of (\ref{eq:2Ddiscretescheme}), $u_{d}^{n}\in W_{d}$ the solution of the reduced FE formulation (\ref{eq:2DPODscheme}). Then we have
\begin{equation}
\left\|u_d^n\right\|^2+ (2\tau)\sum_{i=1}^{n}\left\|u_d^i\right\|^2_w\leq C\tau\sum_{i=1}^{n}\left\|f^i\right\|^2+\left\|u_d^0\right\|^2.
\end{equation}
If taking $L=O(N)$ and the snapshots being taken at uniform intervals, there exists
\begin{equation}\label{error1}
\left\|u_{h}^{n}-u_{d}^{n}\right\|\leq M L\left(\sum_{j=d+1}^{l}\lambda_{j}\right)^{1/2}+M\tau.
\end{equation}
\end{theorem}
\begin{proof}
Taking $v_d=u_d^n$ in  $(\ref{eq:2DPODscheme})$, we get
\begin{equation}
\begin{aligned}
&\left\|u_d^n\right\|^2+\tau\left\|u_d^n\right\|_w^2\leq\tau\left\|f^n\right\|\left\|u_d^n\right\|+\left\|u_d^{n-1}\right\|\left\|u_d^n\right\|\\
&\leq\frac{1}{2}\left(\tau\left\|f^n\right\|^2+\tau\left\|u_d^n\right\|^2\right)+\frac{1}{2}\left(\left\|u_d^{n-1}\right\|^2+\left\|u_d^n\right\|^2\right),
\end{aligned}
\end{equation}
which leads to
\begin{equation}
\left\|u_d^n\right\|^2+ (2\tau)\sum_{i=1}^{n}\left\|u_d^i\right\|^2_w\leq \tau\sum_{i=1}^{n}\left\|f^i\right\|^2+\left\|u_d^0\right\|^2+\tau\sum_{i=1}^{n}\left\|u_d^i\right\|^2.
\end{equation}
According to the discrete Gr\"onwall inequality, we have
\begin{equation}
\left\|u_d^n\right\|^2+ (2\tau)\sum_{i=1}^{n}\left\|u_d^i\right\|^2_w\leq C\tau\sum_{i=1}^{n}\left\|f^i\right\|^2+\left\|u_d^0\right\|^2.
\end{equation}
By (\ref{eq:2Ddiscretescheme}) and (\ref{eq:2DPODscheme}), there exists
\begin{equation}\label{eq:1Derroreq}
\begin{aligned}
    &\left(u^{n}_h-u^{n}_d,v_d\right)+\tau C_{\alpha}\left(~_{0}D_{x}^{\frac{\alpha}{2}}\left(u^{n}_h-u^{n}_d\right),~_{x}D_{1}^{\frac{\alpha}{2}}v_d\right)\\
    &+\tau C_{\alpha}\left(~_{x}D_{1}^{\frac{\alpha}{2}}\left(u^{n}_h-u^{n}_d\right),~_{0}D_{x}^{\frac{\alpha}{2}}v_d\right)+\tau C_{\beta}\left(~_{0}D_{y}^{\frac{\beta}{2}}\left(u^{n}_h-u^{n}_d\right),~_{y}D_{1}^{\frac{\beta}{2}}v_d\right)\\
    &+\tau C_{\beta}\left(~_{y}D_{1}^{\frac{\beta}{2}}\left(u^{n}_h-u^{n}_d\right),~_{0}D_{y}^{\frac{\beta}{2}}v_d\right)=\left(u^{n-1}_h-u^{n-1}_d,v_d\right).
    \end{aligned}
\end{equation}
Combining (\ref{projection2}) with (\ref{eq:1Derroreq}), we have
\begin{equation}
\begin{aligned}
&\left\|P^du_h^n-u_d^n\right\|^2+\tau\left\|~_0D_x^{\frac{\alpha}{2}}(P^du_h^n-u_d^n)\right\|^2+\tau\left\|~_0D_y^{\frac{\beta}{2}}(P^du_h^n-u_d^n)\right\|^2\\
=&\left(P^du_h^n-u_d^n,P^du_h^n-u_d^n\right)\\
&+\tau C_{\alpha}\left(~_0D_x^{\frac{\alpha}{2}}\left(P^du_h^n-u_d^n\right),~_xD_1^{\frac{\alpha}{2}}\left(P^du_h^n-u_d^n\right)\right)\\
&+\tau C_{\alpha}\left(~_xD_1^{\frac{\alpha}{2}}\left(P^du_h^n-u_d^n\right),~_0D_x^{\frac{\alpha}{2}}\left(P^du_h^n-u_d^n\right)\right)\\
&+\tau C_{\beta}\left(~_0D_y^{\frac{\beta}{2}}\left(P^du_h^n-u_d^n\right),~_yD_1^{\frac{\beta}{2}}\left(P^du_h^n-u_d^n\right)\right)\\
&+\tau C_{\beta}\left(~_yD_1^{\frac{\beta}{2}}\left(P^du_h^n-u_d^n\right),~_0D_y^{\frac{\beta}{2}}\left(P^du_h^n-u_d^n\right)\right)\\
=&\left(P^du_h^n-u_h^n,P^du_h^n-u_d^n\right)+\left(u_h^n-u_d^n,P^du_h^n-u_d^n\right)\\
&+\tau C_{\alpha}\left(~_0D_x^{\frac{\alpha}{2}}\left(P^du_h^n-u_h^n\right),~_xD_1^{\frac{\alpha}{2}}\left(P^du_h^n-u_d^n\right)\right)\\
&+\tau C_{\alpha}\left(~_0D_x^{\frac{\alpha}{2}}\left(u_h^n-u_d^n\right),~_xD_1^{\frac{\alpha}{2}}\left(P^du_h^n-u_d^n\right)\right)\\
&+\tau C_{\alpha}\left(~_xD_1^{\frac{\alpha}{2}}\left(P^du_h^n-u_h^n\right),~_0D_x^{\frac{\alpha}{2}}\left(P^du_h^n-u_d^n\right)\right)\\
&+\tau C_{\alpha}\left(~_xD_1^{\frac{\alpha}{2}}\left(u_h^n-u_d^n\right),~_0D_x^{\frac{\alpha}{2}}\left(P^du_h^n-u_d^n\right)\right)\\
&+\tau C_{\beta}\left(~_0D_y^{\frac{\beta}{2}}\left(P^du_h^n-u_h^n\right),~_yD_1^{\frac{\beta}{2}}\left(P^du_h^n-u_d^n\right)\right)\\
&+\tau C_{\beta}\left(~_0D_y^{\frac{\beta}{2}}\left(u_h^n-u_d^n\right),~_yD_1^{\frac{\beta}{2}}\left(P^du_h^n-u_d^n\right)\right)\\
&+\tau C_{\beta}\left(~_yD_1^{\frac{\beta}{2}}\left(P^du_h^n-u_h^n\right),~_0D_y^{\frac{\beta}{2}}\left(P^du_h^n-u_d^n\right)\right)\\
&+\tau C_{\beta}\left(~_yD_1^{\frac{\beta}{2}}\left(u_h^n-u_d^n\right),~_0D_y^{\frac{\beta}{2}}\left(P^du_h^n-u_d^n\right)\right)\\
=&\left(P^du_h^n-u_h^n,P^du_h^n-u_d^n\right)+\left(u^{n-1}_h-u^{n-1}_d,P^du_h^n-u_d^n\right).
\end{aligned}
\end{equation}
According to Cauchy-Schwartz inequality, there exists
\begin{equation}
\|P^du_h^n-u_d^n\|\leq \|P^du_h^n-u_h^n\|+\|u_h^{n-1}-u_d^{n-1}\|.
\end{equation}
Further using triangle inequality,
\begin{equation}
\left\|u_{h}^{n}-u_{d}^{n}\right\|\leq \left\|P^du_h^n-u_h^n\right\|+\left\|P^du_h^n-u_d^n\right\|,
\end{equation}
we get
\begin{equation}\label{eq2d}
\left\|u_{h}^{n}-u_{d}^{n}\right\|\leq 2\left\|P^du_h^n-u_h^n\right\|+\left\|u_h^{n-1}-u_d^{n-1}\right\|.
\end{equation}
Summing (\ref{eq2d}) for $1,2,\cdots,n$, squaring, and using Harmonic inequality, we get
\begin{equation}
\left\|u_{h}^{n}-u_{d}^{n}\right\|^2\leq 4n\sum_{i=1}^{n}\left\|P^du_h^i-u_h^i\right\|^2.
\end{equation}

For $1\leq n\leq N$, we assume $n_i\leq n\leq n_{i+1}\leq N \,(i=1,2,\ldots,L-1)$ and $n_i\leq n\leq\frac{n_i+n_{i+1}}{2}$, then expanding $u_h^n$ into Taylor series about $t_{n_i}$ yields that
\begin{equation}
u_h^n=u_h^{n_i}+\varepsilon_i\tau u_{ht}(\xi_i),  ~~~~~~t_{n_i}\leq \xi_i \leq t_n, ~~i=1,2,\cdots,L,
\end{equation}
where $\varepsilon_i$ is the step number from $t_{n_i}$ to $t_n$. If snapshots are taken at uniform intervals, then $|\varepsilon_i|\leq \frac{N}{2L}$ and
\begin{equation}\label{equ:taylor1}
\left\|u_{h}^{n}-u_{d}^{n}\right\|^2\leq MN\frac{N}{L}\sum_{i=n_1}^{n_i}\left\|P^du_h^i-u_h^i\right\|^2+MN\varepsilon_i^2\tau^2\sum_{i=1}^{n}\left\|P^du_{ht}(\xi_i)-u_{ht}(\xi_i)\right\|^2.
\end{equation}
For the second term of (\ref{equ:taylor1}), we have the following estimate,
\begin{equation}\label{equ:taylor2}
\begin{aligned}
  &\sum_{i=1}^{n}\|P^du_{ht}(\xi_i)-u_{ht}(\xi_i)\|^2\\
  \leq&C\frac{N}{L}\sum_{i=1}^{L-1}\left\|P^d\left(\frac{u^{n_{i+1}}_{h}-u^{n_i}_{h}}{\varepsilon_i\tau}\right)-\frac{u^{n_{i+1}}_{h}-u^{n_i}_{h}}{\varepsilon_i\tau}+O(\tau)\right\|^2\\
  \leq&C\frac{N}{L(\varepsilon_i\tau)^2}\sum_{i=1}^{L-1}\left\|P^d\left(u^{n_{i+1}}_{h}-u^{n_i}_{h}\right)-\left(u^{n_{i+1}}_{h}-u^{n_i}_{h}\right)+O(\tau^2)\right\|^2
\end{aligned}
\end{equation}
According to (\ref{equ:taylor2}) and (\ref{equ:taylor1}), we get
\begin{equation}
  \begin{aligned}
    &\left\|u_{h}^{n}-u_{d}^{n}\right\|^2\\
    \leq& MN\frac{N}{L}\sum_{i=n_1}^{n_i}\left\|P^du_h^i-u_h^i\right\|^2+MN\varepsilon_i^2\tau^2\sum_{i=1}^{n}\left\|P^du_{ht}(\xi_i)-u_{ht}(\xi_i)\right\|^2\\
    \leq& M\frac{N^2}{L}\sum_{i=n_1}^{n_i}\left\|P^du_h^i-u_h^i\right\|^2+O(\tau^2)
  \end{aligned}
\end{equation}
If taking $L=O(N)$, according to fractional Poincare inequality and Theorem \ref{prosp}, there exists
\begin{equation}
\left\|u_{h}^{n}-u_{d}^{n}\right\|\leq M L\left(\sum_{j=d+1}^{l}\lambda_{j}\right)^{1/2}+M\tau.
\end{equation}
\end{proof}
\begin{theorem}
Let $u^n$ be the exact solution of Eq. (\ref{equation2D}) and $u_{d}^{n}$  the solution of the reduced formulation (\ref{eq:2DPODscheme}). Then we have
\begin{equation}
\left\|u^n-u_{d}^{n}\right\|\leq M L\left(\sum_{j=d+1}^{l}\lambda_{j}\right)^{1/2}+M\tau+M h^{k+1-\gamma},
\end{equation}
where $\gamma=\max(\alpha,\beta)$.
\end{theorem}
\begin{proof}
According to \cite{Bu2014}, we have
\begin{equation}
  \left\|u^n-u_{h}^{n}\right\|\leq C\left(\tau+h^{k+1-\gamma}\right).
\end{equation}
Combining Theorem \ref{thm2Derror} and the triangle inequality leads to the desired result.
\end{proof}
\begin{remark}
Since the term $L\left(\sum\limits_{j=d+1}^{l}\lambda_{j}\right)^{1/2}$ is produced by reduced-order and the error formula, it implies that $L$ isn't too large. At the same time, the error estimate provides a measurement for deciding the number of POD bases needed, namely, the number $d$ of POD bases should satisfy $L\left(\sum\limits_{j=d+1}^{l}\lambda_{j}\right)^{1/2}\leq max\{\tau, h^{k+1-\gamma}\}$ when $L=O(N)$ for obtaining the optimal convergence order.
\end{remark}

\section{Numerical experiments}
In this section, we verify the effectiveness of the algorithm and show the advantage of the reduced POD FE formulation by numerical examples.
\begin{example}\label{example1}
Consider the exact solution of (\ref{equation2D}) as follows
\begin{equation}
u=4\cos(1.5\pi/2)\cos(1.6\pi/2)\exp(-t)\sin(2\pi x)^2\sin(2\pi y)^2.
 \end{equation}
 We take $\Omega=[0,1]\times[0,1]$, $\alpha=1.5$, $\beta=1.6$, and $T=1$. The source term $f$ can be calculated numerically. Firstly, we divide the field $\Omega$ into 256 squares with side length $\bigtriangleup x= \bigtriangleup y=1/16$, and then link the diagonal of the square to divide each square into two triangles in the same direction which constitute triangulation $\{\Im_{h}\}$, and the time step size is taken as $\tau=1/256$.

 A group of numerical solutions are obtained by the classical finite element, and then 17 snapshots are chosen at an uniform intervals from 256 transient solutions. Figure \ref{fig:2DT1stsumoflambda} shows the change trend of $\sum_{i=d+1}^{17}\lambda_i$ as the number $d$ of POD basis increases; combining with the theoretical error estimate, we only need 1 POD basis to satisfy the requested accuracy. When $T=1$, we find that $16\times 16\times 2=512$ degrees of freedom are needed and the required computing time is about 2.58 seconds for the usual finite element method; while for the reduced FE formulation, 1 degree of freedom is needed and the corresponding time is about 0.043 seconds, which shows that our method can save memory and computing time effectively. Figure \ref{fig:2DT1stPODsol} and \ref{fig:2DT1streal} depict the POD solution and the real solution graphically when $T=1$, respectively, and it can be found that the POD solution is visually the same as the real solution. The numerical results confirm the  theoretical analysis.

\begin{figure}[h]
  \begin{center}

  \includegraphics[width=6.4cm,height=4.8cm,angle=0]{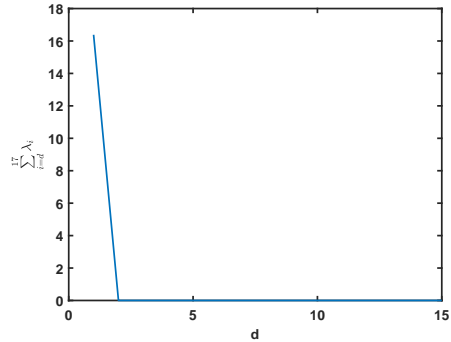}
  \caption{Sum of $\lambda_i$}\label{fig:2DT1stsumoflambda}
  \end{center}
\end{figure}

\begin{figure}[h]
  \begin{center}

  \includegraphics[width=6.4cm,height=4.8cm,angle=0]{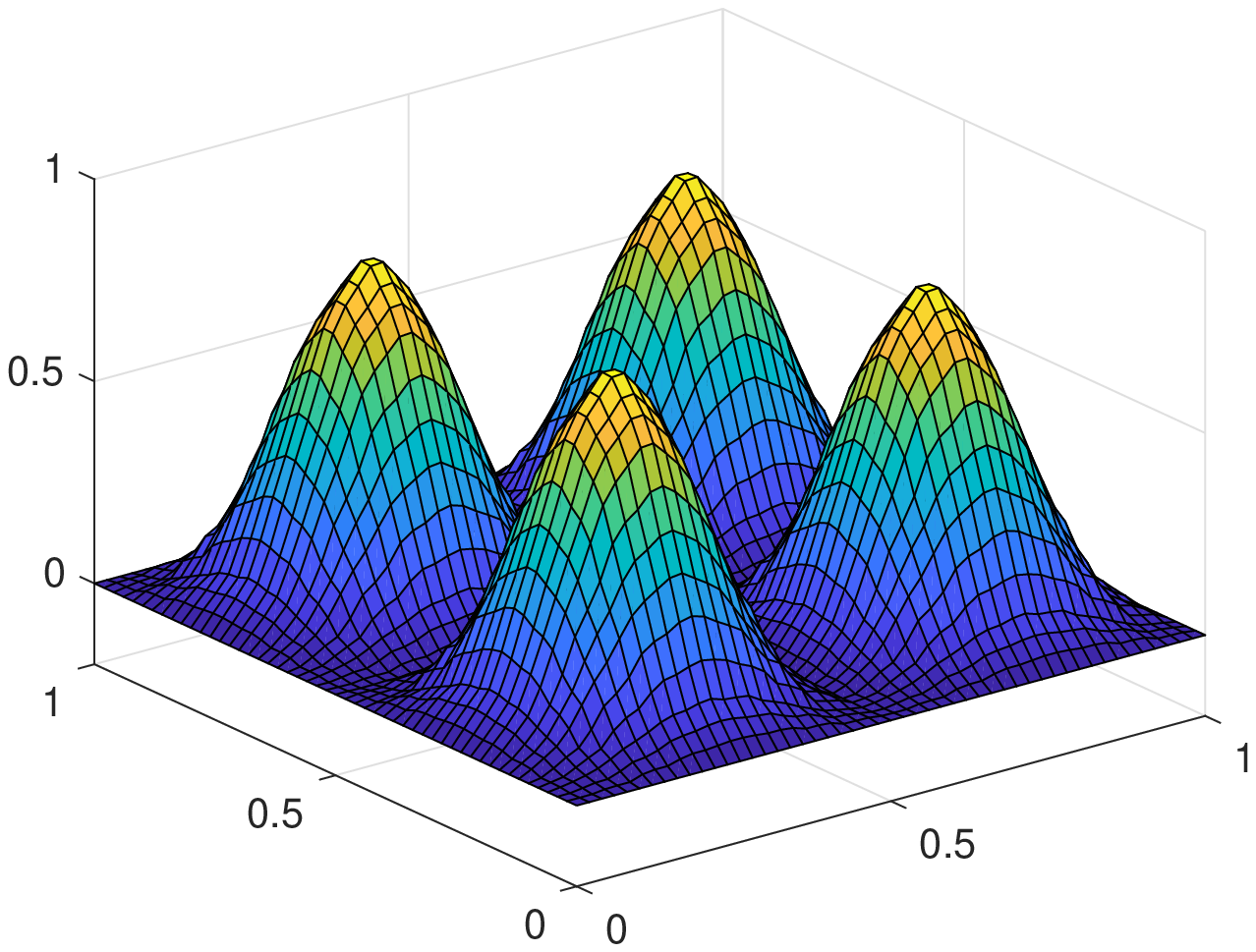}
  \caption{POD solution when $T=1$}\label{fig:2DT1stPODsol}
  \end{center}
\end{figure}

\begin{figure}[h]
  \begin{center}

  \includegraphics[width=6.4cm,height=4.8cm,angle=0]{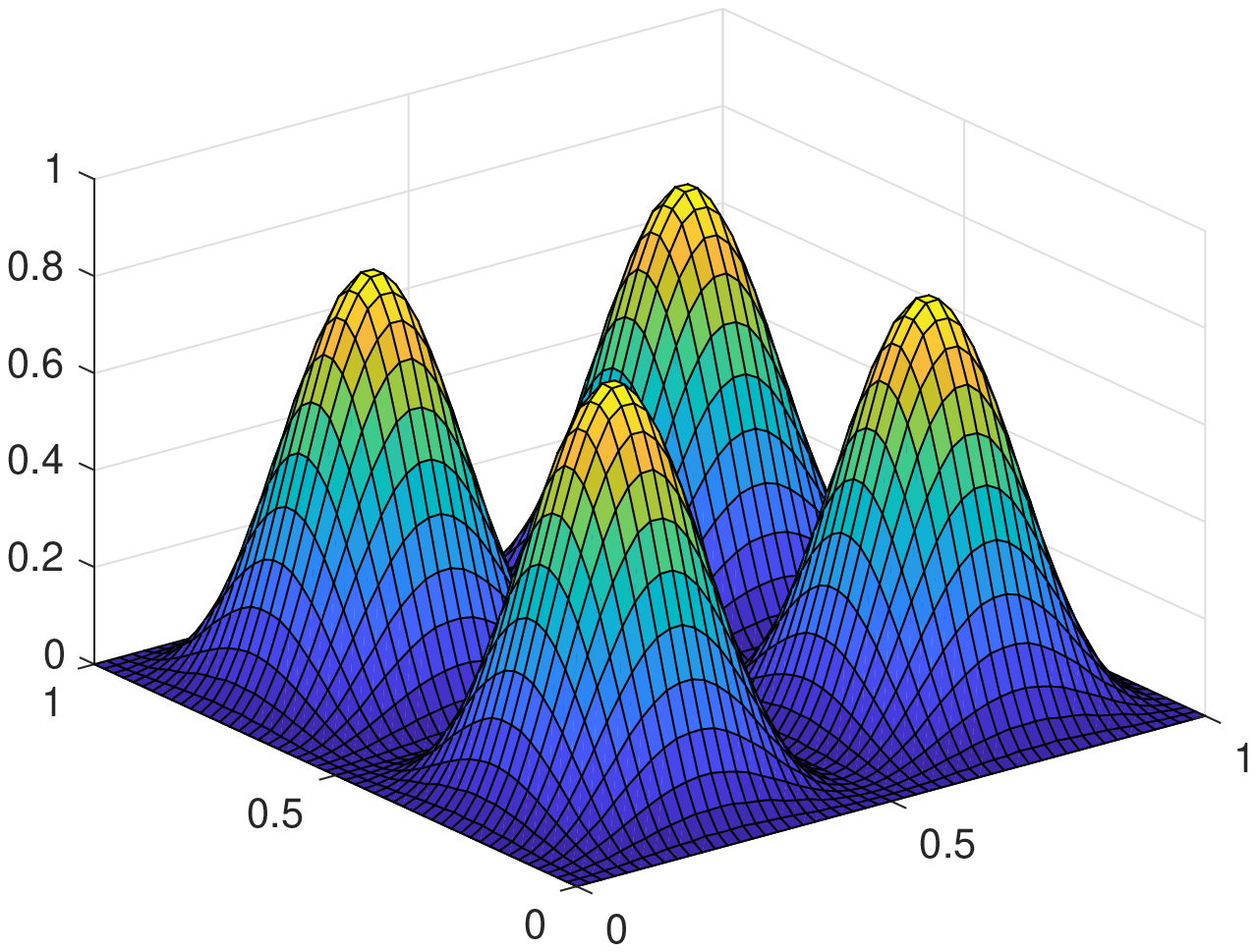}
  \caption{Real solution when $T=1$}\label{fig:2DT1streal}
  \end{center}
\end{figure}
\end{example}
Next, we give an exact solution whose shape changes sharply with the time evolution. In this case, usually the adaptive finite element method is adopted to do simulation. Here we show that the reduced FE method can also do it very well.

\begin{example}
The exact solution of (\ref{equation2D}) is taken as
\begin{equation}
u=4\times 10^3\cos(1.5\pi/2)\cos(1.6\pi/2)\exp\left(-\frac{(x-t)^2+(y-t)^2}{0.04}\right)x^2(x-1)^2y^2(1-y)^2.
 \end{equation}
 Here, we take $\Omega=[0,1]\times[0,1]$, $\alpha=1.5 $, $\beta=1.6$, and $T=1$. The mesh is generated by the same way as Example \ref{example1} and the time step size is $1/256$. The source term can be calculated numerically.

The finite element solutions are obtained with $h=1/16$ and $\tau=1/256$. We choose 34 values from 256 values; every 7 values compose of a set of snapshots. Figure \ref{fig:2DT1time} presents the CPU time of the general FE scheme and the reduced FE system; it can be found that the computation time can be reduced significantly by using the reduced system. Figure \ref{fig:2DT1error} shows the change of the error of the POD solution when using different number $d$ of POD bases; it can be noted that the error decreases as the number $d$ of POD bases increases. Figure \ref{fig:2DT1sumoflambda} shows the trend of $\sum_{i=d}^{34}\lambda_i$ as the number $d$ of POD basis increases, being consistent with the theoretical estimate. Furthermore as error estimate predicts, by combining Figures \ref{fig:2DT1error} and \ref{fig:2DT1sumoflambda}, one can notice  that the error depends on $\sum_{i=d+1}^{34}\lambda_i$.
%
%
 Figures \ref{fig:2DT1PODSol} and \ref{fig:2DT1real} depict the POD solution and the real solution graphically when $T=1$; from them, it can be noted that the POD solution is visually the same as the real solution, which shows the effectiveness of the provided algorithm.
\begin{figure}[h]
  \begin{center}

  \includegraphics[width=6.4cm,height=4.8cm,angle=0]{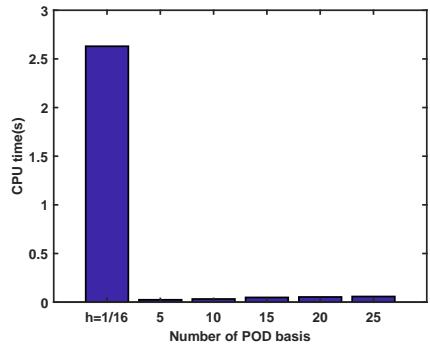}
  \caption{Used CPU Time}\label{fig:2DT1time}
  \end{center}
\end{figure}

\begin{figure}[h]
  \begin{center}

  \includegraphics[width=6.4cm,height=4.8cm,angle=0]{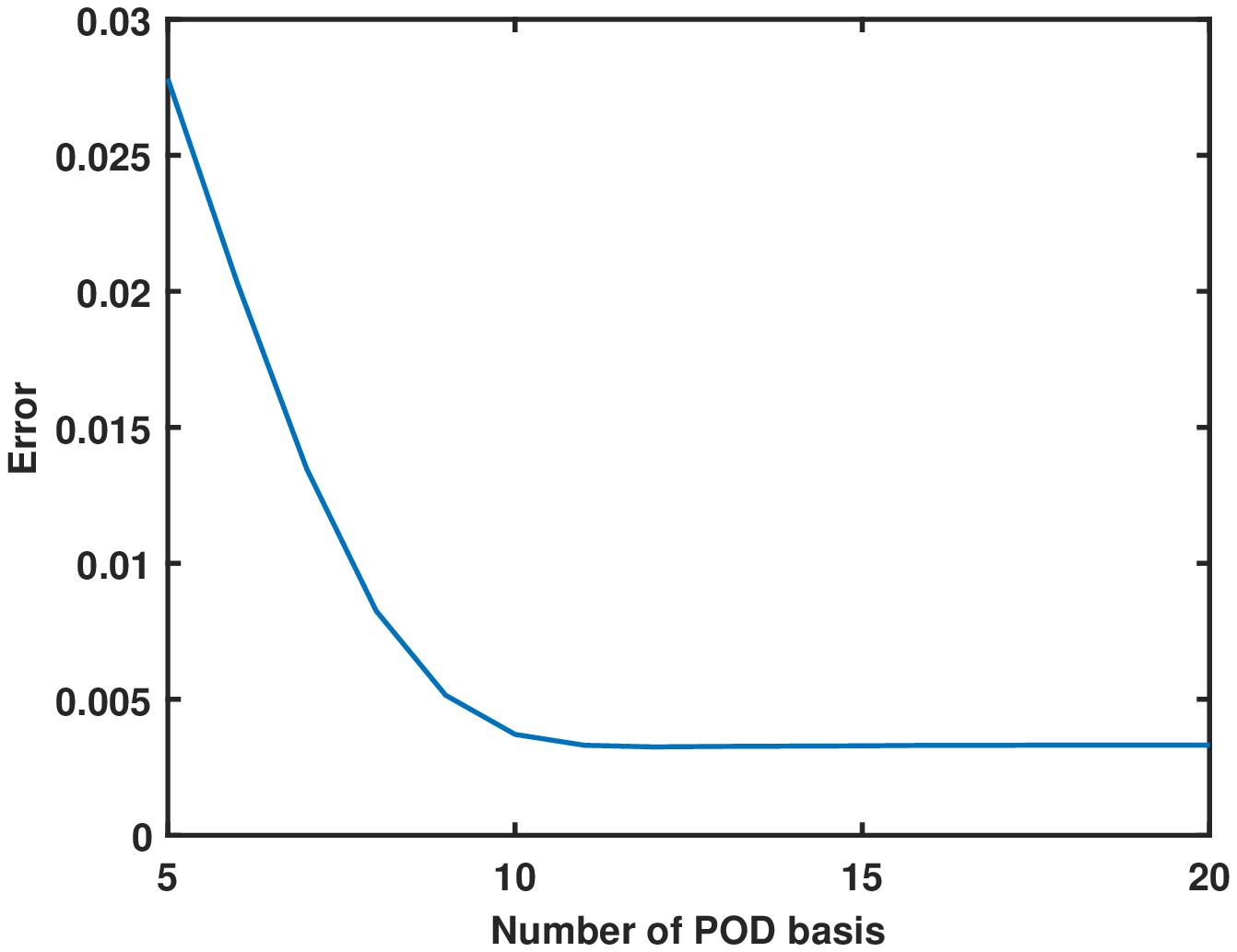}
  \caption{Errors}\label{fig:2DT1error}
  \end{center}
\end{figure}

\begin{figure}[h]
  \begin{center}

  \includegraphics[width=6.4cm,height=4.8cm,angle=0]{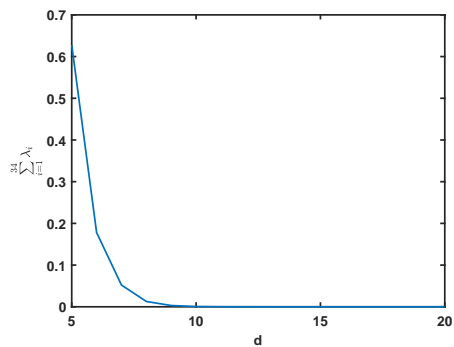}
  \caption{Sum of $\lambda$}\label{fig:2DT1sumoflambda}
  \end{center}
\end{figure}

\begin{figure}[h]
  \begin{center}

  \includegraphics[width=6.4cm,height=4.8cm,angle=0]{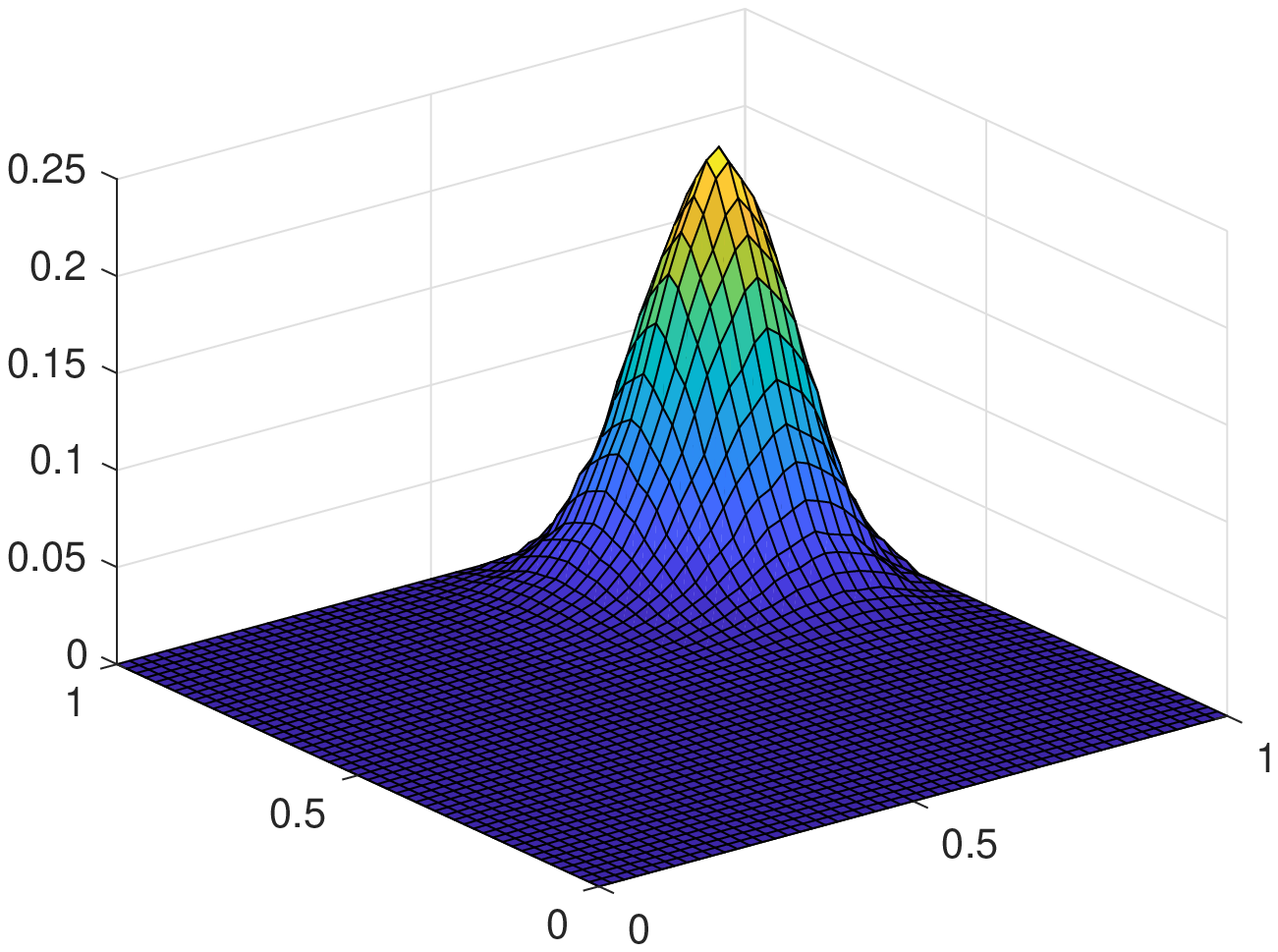}
   \caption{POD solution when $T=1$}\label{fig:2DT1PODSol}
  \end{center}
\end{figure}
\end{example}

\begin{figure}[h]
  \begin{center}

  \includegraphics[width=6.4cm,height=4.8cm,angle=0]{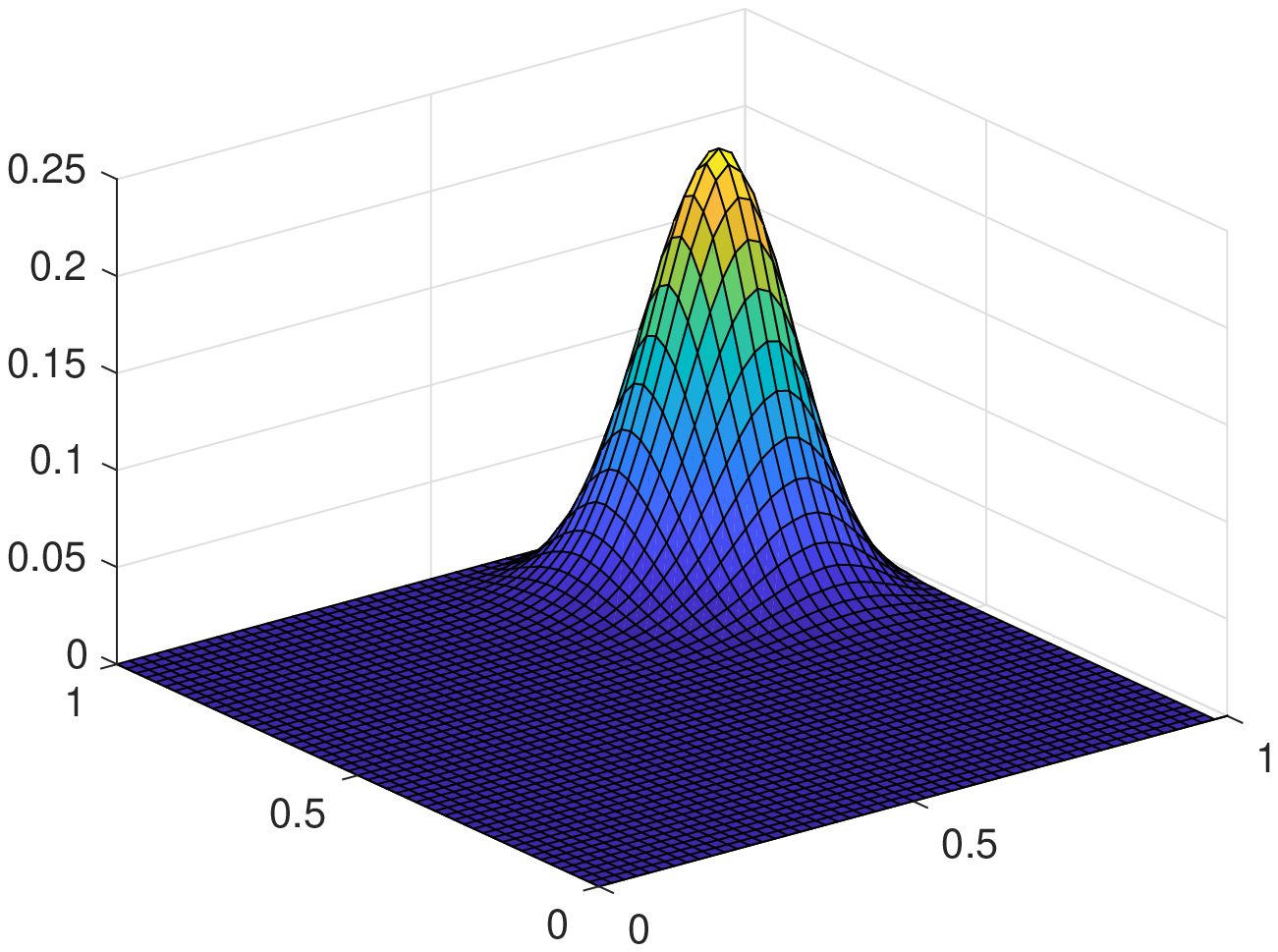}
  \caption{Real solution when $T=1$}\label{fig:2DT1real}
  \end{center}
\end{figure}

\section{conclusion}

This paper provides the basic framework for solving space FPDEs by reduced FE model. The basic strategy of choosing the reduced basis is provided. The detailed numerical stability analysis and error estimates are proposed for the reduced model. To show the effectiveness of the reduced model in alleviating computational load, saving memory, and keeping accuracy, extensive numerical experiments are performed, which also confirm the theoretical results.
 In the future study, we will use the reduced FE method to solve large-scale models to reduce the memory storage requirements.

\section*{Acknowledgments}
 This work was supported by the National Natural Science Foundation of China under Grant No. 11671182, and the Fundamental Research Funds for the Central Universities under Grant No. lzujbky-2017-ot10.


\end{document}